\documentclass{amsart}
\usepackage{amssymb}
\usepackage{amsmath,amsfonts,amsthm,amstext}
\newtheorem{theorem}{Theorem}[section]
\newtheorem{lemma}[theorem]{Lemma}

\newtheorem{prop}[theorem]{Proposition}
\newtheorem{cor}[theorem]{Corollary}
\newtheorem{proposition}[theorem]{Proposition}
\numberwithin{equation}{section}

\def\Vightarrow#1{\smash{\mathop{\to}\limits^{#1}}}

\begin{document}

\title[Rank and deficiency gradients of Thompson groups]{Rank and deficiency gradients of generalised Thompson groups of type $F$}

\author{Dessislava H. Kochloukova}

\begin{abstract} For an arbitrary sequence $(G_s)$ of subgroups of finite index in the generalised Thompson group $$F_{n, \infty} = \langle x_0, x_1, \ldots , x_m, \ldots \mid x_i^{x_j} = x_{i+ n-1} \hbox{ for } i > j \geq	 0 \rangle$$
it is shown that $\sup_{s \geq 1} d(G_s) < \infty$ and that the deficiency gradient of $F_{n, \infty}$ with respect to $(G_s)$  is 0 provided $[G : G_s]$ tends to infinity. A higher dimensional analogue is considered for $n = 2$.
\end{abstract} 

\maketitle

\section{Introduction}
In this paper we apply methods from $\Sigma$ theory to study the minimal number of generators  and the deficiency of subgroups of finite index in the generalized Thompson groups of type $F$. 
The generalised Thompson groups are defined by the infinite presentation 
$$F_{n, \infty} = \langle x_0, x_1, \ldots , x_m, \ldots \mid x_i^{x_j} = x_{i+ n-1} \hbox{ for } i > j \geq	 0 \rangle,$$ where $n \geq 2$. In the case $n = 2$ we get the classical Richard Thompson group $F$. In is not difficult to see that $F_{n, \infty}$  is finitely presented but actually much more holds i.e.     Brown and Geoghegan showed in \cite{BG} that $F$ has type $FP_{\infty}$  and later Brown proved in \cite{Brown2} that the same holds  for the group $F_{n, \infty}$. Furthermore it was shown in \cite{Brown2}  that
 the commutator subgroup $F_{n, \infty}'$ is simple and hence any subgroup of finite index in $F_{n, \infty}$ contains the commutator and hence is normal in $F_{n, \infty}$. Thought it is still an open problem whether the group $F$ is amenable it was  shown by Brin and Squier that $F$ does not contain non-cyclic free subgroups \cite{B-S}. The same holds for $F_{n, \infty}$ ( note that  $F_{n,\infty}$ embeds in $F$ and $F$ embeds in $F_{n, \infty}$).

The group $F_{n, \infty}$ has a presentation as a group of PL transformations of a closed interval \cite{essay} or the real line \cite{BrinGuzman}. 
In the classical case ($n = 2$) the automorphism group of $F = F_{2, \infty}$ was described by Brin \cite{Brin}
  but in the case $n \geq 3$ as shown by Brin and Guzm\'an $Aut(F_{n, \infty})$ does not behave as $Aut(F)$ i.e. $Aut(F_{n, \infty})$  is wild in the sense that not every automorphism of $F_{n, \infty}$ comes from conjugation with a PL automorphism of an interval (respectively the real line if the presentation uses the real line and not a closed interval).

Let $H$ be a finitely generated group. A chain  $( H_s  )$ in $H$ is a sequence $\ldots < H_{s+1} < H_s < \ldots < H_1 = H$ 
 of subgroups of finite index in $H$  such that $H_{s+1}$ is a proper subgroup of $H_s$ for every $s$. In \cite{Lackenby}  M. Lackenby  defined the rank gradient of $H$ with respect to the chain $(H_s)$ as
$$
RG(H, ( H_s )) = d(H_s) - 1 / [H : H_s],
$$ 
where $d(H_s)$ is the minimal number of generators of $H_s$.
For a finitely presented infinite amenable group $H$  and for any chain $(H_s)$ Ab\'ert, Jaikin-Zapirain and Nikolov  showed  that  $RG(H, ( H_s )) = 0$ \cite{ANJ} but as we mentioned before  
for the generalized Thompson groups $F_{n, \infty}$ it is still an open problem whether there are amenable. Note that as any chain of finite index subgroups  in $F_{n, \infty}$ is normal, condition 1) from \cite[Thm.~3]{ANJ} is equivalent to property $(\tau)$ and as the abelianization $\mathbb{Z}^n$ of $F_{n, \infty}$  does not have property ($\tau$) with respect to any  chain we have that $F_{n, \infty}$ does not have property $(\tau)$. Furthermore
$F_{n, \infty}$ is not virtually a free product with amalgamation ($F_{n, \infty}$ does not contain free non-cyclic subgroups). Then by \cite[Thm.~3]{ANJ} for any chain $(G_s)$ in $F_{n, \infty}$ we have $RG(F_{n, \infty}, (G_s)) = 0$.
Our first result generalizes the fact that the rank gradient of $F_{n, \infty}$ with respect to any chain is 0.

\medskip
{\bf Theorem A1 } {\it Let $\mathcal A$ be the set of all subgroups of finite index in $F_{n, \infty}$. Then
$sup_{H \in {\mathcal A}} \ d(H) < \infty$.}

\medskip
 The proof of Theorem A1 relies on the fact that there are only two subgroups $G' < K < G$ such that $G / K \simeq \mathbb{Z}$ and $K$ is not finitely generated. Furthermore these two groups $K$ are conjugated by external automorphism of $G$. The classification of the finitely generated subgroups containing the commutator of a finitely generated group $H$ is given by the Bieri-Neumann-Renz-Strebel invariant $\Sigma^1(H)$, discussed in details in the preliminaries. The invariant $\Sigma^1(H)$ is part of a sequence of invariants $\{ \Sigma^i(H) \}_{
i\leq m}$ defined for groups $H$ of homotopical type $F_m$. In the case of groups $H$ of PL transformations of a closed interval under some mild conditions on $H$ the invariant $\Sigma^1(H)$ was calculated by Bieri, Neumann, Strebel in \cite{BNS}. In particular this gives a complete description of $\Sigma^1(F_{n,\infty})$. The invariant $\Sigma^m(F)$ for arbitrary $m \geq 2$ was calculated by Bieri, Geoghegan, Kochloukova in \cite{BGK} and $\Sigma^2(F_{n, \infty})$ was calculated for an arbitrary $n \geq 3$ by Kochloukova in \cite{K}.

In the case $n = 2$  Theorem A1 generalizes to higher dimensions. We do not know whether a version of Theorem A2 holds for $F_{n, \infty}$. The proof of Theorem A2 uses significantly that $F/ F'$ has small torsion-free rank.

\medskip
{\bf Theorem A2} {\it Let $H$ be a subgroup of finite index in $F$. Then there is a $K(H, 1)$ complex with $r(H, j)$ cells in dimension $j$ such that  $r(H, j) \leq 8j - 4$ for $j \geq 3$, $r(H,2) = 12, r(H, 1) = 5 $  and $r(H, 0) = 1$.}

\medskip

For a finitely presented group $H$ the deficiency gradient associated to a chain $(H_s)$ in $H$ is
$$
DG(H, ( H_s ) ) = \lim_{s \to \infty} def (H_s) / [ H : H_s],
$$
where $def(H_s)$ is the deficiency of $H_s$. Here we fix the deficiency of a finitely presented group as the suprimum of the number of generators minus the number of relations over all possible finite presentations of the group.
In the case when every $H_s$ is normal and $\cap_s H_s = 1$ the deficiency gradient  was defined in \cite{BK}.  Recently the same invariant was considered by Abert and Gaboriau in 
\cite{AG} and by Kar and Nikolov in \cite{KN}. In \cite{KN} Kar and Nikolov show that for every finitely presented residually finite amenable group  $H$ and for any normal chain $(H_s)$ we have $DG(H, (H_s)) = 0$. 
Though $F_{n, \infty}$ is not residually finite and we do not know whether it  is amenable by the following theorem the same property holds for $F_{n,\infty}$. Furthermore the condition on inclusions in the definition of a chain  is relexed by assuming only that the index goes to infinity.

\medskip
{\bf Theorem B} {\it Let $( G_s )$ be a sequence of subgroups of finite index in $G = F_{n, \infty}$ such that the index $[G : G_s]$ tends to infinity when $s$ tends to infinity. Then $DG(G,  ( G_s )) = 0$.}

\medskip
Finally we define a new invariant $\chi_m(H)$ that behaves as a partial Euler characteristic and set for a sequence $(H_s)$ of subgroups of finite index in  $H$  $m$-dimensional  gradient by 
$$
\chi_mG(H, ( H_s ) ) = \lim_{s \to \infty} \chi_m (H_s) / [ H : H_s].
$$
By definition for a group $H$ of type $F_m$ we set
$$\chi_m(H) =   inf \{ \sum_{0 \leq i \leq m} (-1)^{m - i} \alpha_i\}$$
where the infimum is over all 
 $K(H,1)$ CW-complexes where $\alpha_j$ is the number of cells in dimension $j \leq m$. We show in
Section \ref{final0} that  for $H$ a subgroup of finite index in  $F_{n, \infty}$ we have $
\sum_{0 \leq i \leq m} (-1)^{m - i} \alpha_i \geq 0$ for every $m$.

\medskip
{\bf Theorem C} {\it Let $ (G_s )$ be a sequence of subgroups of finite index in $G = F$ such that the index $[G : G_s]$ tends to infinity when $s$ tends to infinity.
Then $\chi_mG(F, (G_s)) = 0$.}

\medskip
Theorem B is a corollary of Proposition \ref{sigma} and Theorem C is a corollary of Theorem A2, as well is a corollary of Propostion \ref{sigma}. The condition (\ref{sigmacondition})  on $\Sigma^m(G)$ in the statement of Propostion \ref{sigma} i.e.
\begin{equation} \label{nnn}
\Sigma^m(G) = S(G) \setminus  conv_{\leq 2} \{ [\chi_1], [\chi_2] \},\hbox{ where } \Sigma^1(G) = S(G) \setminus \{ [\chi_1], [\chi_2] \} 
\end{equation}
holds in some particular cases (see the preliminaries on Sigma invariants). We strongly believe that (\ref{nnn}) holds for general $m$ but this is open for now. If (\ref{nnn})  holds for $G = F_{n, \infty}$  then Theorem C holds for $G = F_{n, \infty}$.

\medskip
Acknowledgements : I would like to thank  A. Kar and N. Nikolov for sending their preprint \cite{KN}  and  N. Nikolov and M. Abert for the encouragement to start writing this paper during the conference ``Golod-Shafarevich Groups and Algebras, and the Rank Gradient", Vienna 2012 .

\section{Preliminaries on $\Sigma$ theory}

The invariant  $\Sigma^m(H)$ is defined for any group $H$ of homotopical type $F_m$ i.e. there is a $K(H, 1)$ with finite $m$-skeleton.  By definition $\Sigma^m(H)$ is a subset of the character sphere $S(H)$, where $S(H) = Hom (H, \mathbb{R}) \setminus \{ 0 \} / \sim$ and $\sim$ is the equivalence relation given by $\chi_1 \sim \chi_2$ if and only if there is a positive real number $r$ such that $\chi_1 = r \chi_2$. We write $[\chi]$ for the equivalence class of $\chi$.  

Let $\Gamma$ be the $m$-skeleton of the universal cover of some  $CW$-complex that is a $K(H,1)$ with finite $m$-skeleton and one vertex, thus the set of vertices of $\Gamma$ is $H$. In the case $m = 1$ $\Gamma$ is the Cayley graph of $H$ with respect to a finite generating set. We consider a version of the Cayley graph and its higher dimensional analogies where $H$ acts freely on the left. Thus in this paper all modules and group actions (if not otherwise stated) are left ones.  

Let $\chi : H \to \mathbb{R}$ be a non-zero homomorphism. Define $\Gamma_{\chi \geq r}$ as the $CW$-subcomplex of $\Gamma$ spanned by the vertices in $H_{\chi \geq r} = \{ g \in H | \chi(g) \geq r \}$. Note that for $d \leq 0$ the inclusion map $\Gamma_{\chi \geq 0} \to \Gamma_{\chi \geq d}$ induces a map  $\pi_{i}(\Gamma_{\chi \geq 0}) \to \pi_{i}(\Gamma_{\chi \geq d}) $. By definition
$$
\Sigma^m(H) = \{ [\chi] \in S(H) | \hbox{ there is } d = d_{\chi}  \leq 0 $$ $$ \hbox{ such that the map } \pi_{i}(\Gamma_{\chi \geq 0}) \to \pi_{i}(\Gamma_{\chi \geq d}) \hbox{ is trivial } \hbox{ for all } i \leq m-1 \}.
$$
The case $m = 1$ is special, in this case for $[\chi] \in \Sigma^1(H)$ the constant $d_{\chi}$ can be chosen always 0 i.e. $\Gamma_{\chi \geq 0}$ is connected.

\begin{theorem} \label{preli1} \cite{BieriRenz}, \cite{Renz}   Let $H$ be a group of type $F_m$ (resp. $FP_m$) and $K$ be a subgroup of $H$ that contains the commutator. Then $K$ is of type $F_m$ (resp. $FP_m$) if and only if for every non-zero homomorphism $\chi : H \to \mathbb{R}$ such that $\chi(K) = 0$ we have $[\chi] \in \Sigma^m(H)$ (resp. $[\chi] \in \Sigma^m(H, \mathbb{Z})$).
\end{theorem}

In the case of groups of PL transformations of a closed interval under some mild conditions on the group the invariant $\Sigma^1$ was calculated in \cite{BNS}.
Applying this to the generalised Thompson groups we have

\begin{theorem}  \label{preli2} \cite[Prop.~9]{K} For $G = F_{n, \infty}$ we have $\Sigma^1(G) = S(G) \setminus \{ [\chi_1], [\chi_2] \}$, where $\chi_1(x_0) =  - 1, \chi_1(x_i) = 0$ for $i \geq 1$ and $\chi_2(x_i) = 1$ for all $i \geq 0$.
\end{theorem}

This was extended to dimension 2 in \cite{K} and to arbitrary dimension for $n = 2$ in \cite{BGK}.

\begin{theorem}  \label{preli3} \cite[Thm.~A]{K}, \cite[Thm.~A]{BGK} For $G = F_{n, \infty}$ we have $$\Sigma^2(G) = S(G) \setminus conv_{\leq 2} \{ [\chi_1], [\chi_2] \},$$ where $\chi_1(x_0) =  - 1, \chi_1(x_i) = 0$ for $i \geq 1$ and $\chi_2(x_i) = 1$ for all $i \geq 0$, where $conv_{\leq 2} S$ denotes the spherical convex hull of at most 2 elements in $S \subset S(G)$. 

In the case $n = 2$ we have for $m \geq 2$
$$
\Sigma^m(F) = \Sigma^2(F) =  S(F) \setminus conv_{\leq 2} \{ [\chi_1], [\chi_2] \}.
$$
\end{theorem}

Remark. Here we apply  $conv_{\leq 2} S$ for a set $S$ with 2 elements, so $conv_{\leq 2}$ is the spherical convex hull. 

\section{Preliminaries on $Aut(F_{n,\infty})$} \label{automorphisms}

By \cite[Section~3]{K} there are $\varphi , \mu \in Aut(G)$ such that $$\varphi(x_0) = x_0, \varphi(x_i) = x_{i+1} \hbox{ for }i \geq 1$$
and
$$ \mu(x_0) = x_0^{-1},
\mu(x_i) = \varphi^{-2i-1} (x_i) x_0^{-1} \hbox{ for } i \geq 1.$$
Note that $\mu$ has order 2 and $\varphi$ has infinite order. The automorphism $\mu$ of $G$  is given by conjugation with the map $t \to 1- t$ if $G$ is presented as a group of PL homeomorphisms of the interval $[0,1]$ (alternatively if as in \cite{BrinGuzman} $G$ is presented as a group of PL homeomorphisms of the real line then $\mu$ is given by conjugation with the map $t \to - t$) and
\begin{equation} \label{ovo40}
\chi_2 = \chi_1 \mu : G \to \mathbb{R}
\end{equation}

The automorphism  $\varphi$ induces a map $\varphi_0 : G/ G' \to G / G'$ such that $\varphi_0(x_0) = x_0$, $\varphi_0(x_i)  = x_{\rho_0(i)}$ for $ 1 \leq i \leq n-1$ and $\rho_0$ is the cyclic permutation $(1,2, \ldots, n-1)$. The  automorphism  $\mu$ induces a map $\mu_0 : G/ G' \to G / G'$ such that $\mu_0(x_0) = x_0^{-1}$ and $\mu_0(x_i) = x_{\delta(i)} x_0^{-1}$ for $ 1 \leq i \leq n-1$, with $\delta(i) = \rho_0^{-2i-1}(i) = \rho_0^{-i-1} (\rho_0^{-i} (i)) = \rho_0^{-i-1}(n-1)$. Hence $\delta$ exchanges $i$ with $n-i-2$ for $1 \leq i \leq n-3$ and exchanges $n-1$ with $n-2$.

Let  $\widetilde{\varphi}$ be the element of $End_{\mathbb Z}(Hom(F_{n, \infty}, \mathbb{R}))$ induced by $\varphi$ i.e. $\widetilde{\varphi} (\chi) = \chi \varphi$. We fix an isomorphism between  $End_{\mathbb Z}(Hom(F_{n, \infty}, \mathbb{R}))$ and $\mathbb{R}^n$ sending $\chi$ to  $(\chi(x_0), \chi(x_1),$ $ \ldots, \chi(x_{n-1}) )$.
 Then $\widetilde{\varphi}$ has a matrix
\[  A = \begin{pmatrix} 1  &  0  & 0 & 0 & \cdots &   0 & 0 \\
                    0 &0&1& 0 & \cdots &0 & 0\\
                    0 &0& 0& 1 & \cdots &  0 &0  \\
                    0 & 0 & 0 & 0 & \cdots & 0 &0 \\
                    \cdots & \cdots  & \cdots & \cdots   & \cdots & \cdots \\
                    0 & 0 & 0 & 0 & \cdots & 0 &1 \\
                    0& 1 &0& 0&\cdots &0 &0 \end{pmatrix}\] 
Let $\widetilde{\mu}$ be the element of $End_{\mathbb Z}(Hom(F_{n, \infty}, \mathbb{R}))$ induced by $\mu$ i.e. $\widetilde{\mu} (\chi) = \chi \mu$. 
Then by the description of $\mu_0$ we get that $\widetilde{\mu}$ 
 has a matrix             
\[  C = \begin{pmatrix} -1  &  0  & 0 &\cdots  &0 & 0 &   0 & 0 \\
                    -1 &0& 0 &\cdots  &0 & 1&0 & 0\\
                     -1 &0& 0 & \cdots & 1 & 0 &  0 &0  \\
                    \cdots & \cdots  & \cdots & \cdots & \cdots   & \cdots & \cdots \\     
-1 & 0 & 1 & \cdots  &0 & 0 & 0 &0 \\
-1 & 1 & 0 & \cdots &0 & 0 & 0 &0 \\
                    -1 & 0 & 0 & \cdots  & 0 & 0 & 0 &1 \\
                    -1& 0 &0& \cdots & 0  &0 &1 &0 \end{pmatrix}\] 

Note that $\widetilde{\mu}$ has order 2 and $\widetilde{\varphi}$ has order $n$.
Let $\Gamma$ be  the  subgroup of $Aut(G)$ generated by $\varphi$ and $\mu$
and let $D$ be the subgroup of $Aut(Hom(G, \mathbb{R}))$ induced by $\Gamma$, thus $D = \langle \widetilde{\varphi}, \widetilde{\mu} \rangle$.

\begin{proposition} \cite[Prop.~6]{K} Let $G = F_{n, \infty}$. Then
  $D$   induces a group of permutations of $S(G)$ that permutes the elements of  $X$,  where $X$ is one of the following sets : $\Sigma^m(G)$,   $\Sigma^m(G)^c$, $\Sigma^m(G, \mathbb{Z})$ and $\Sigma^m(G, \mathbb{Z})^c$. 
\end{proposition}

\section{Preliminaries on $K(G, 1)$ CW-complexes}

The first lemma is an easy consequence of the Bass-Serre theory. For details see \cite{BK}.

\begin{lemma} \label{Ross}  \cite[Lemma~3.2]{BK} Let $G$ be the fundamental group of a finite graph of groups with underlying graph $\Gamma$  with vertex set $V(\Gamma)$ and edge set $E(\Gamma)$ such that for every $c \in \Gamma = V(\Gamma) \cup E(\Gamma)$  the corresponding group $G_c$ is of type $F_{m - dim(c)}$  and there is  
a $K(G_c, 1)$ complex with $r(G_c,j)$ cells in dimension $j$, where $0 \leq j \leq m -dim(c)$.
Then there is 
a $K(G,1)$ complex with $r(G, j)$ cells in dimension $ 0 \leq j \leq m$, where $r(G, j) = \sum_{c \in \Gamma}  r(G_c, j - dim(c)) $ for $0 \leq j \leq m$.
\end{lemma}

Corollary \ref{stackcor} is a  direct consequence of the stack construction in \cite[Chapter~7]{Rossbook} and \cite[Thm.~7.1.10]{Rossbook}. We state  first  \cite[Thm.~7.1.10]{Rossbook}.

\begin{theorem} \cite[Thm.~7.1.10]{Rossbook} Let $N \ \Vightarrow{} \ G \ \Vightarrow{\pi} \ Q$ be a short exact sequence of groups. Furthermore let $X$ be a $K(N,1)$ CW complex and 
$Z$ be a $K(Q, 1)$ CW-complex. Then there is a $K(G, 1)$ CW complex $Y$ and a cellular map $\varphi : Y \to Z$ with the following properties:

1. $\varphi$ induces the homomorphism $\pi : G \to Q$ on the level of fundamental groups;

2. $\varphi$ is a stack with fibre $X$.
\end{theorem}

\begin{cor} \label{stackcor} Let $N \to G \to Q$ be a short exact sequence of groups with both $N$ and $Q$ of type $FP_m$. Suppose that there is a $K(N,1)$ CW-complex $X$ with $\alpha_j$ cells in dimension $j \leq m$ and there is a $K(Q,1)$ CW-complex $Y$ with $\beta_j$ cells in dimension $j \leq m$. Then there is a $K(G, 1)$ complex with $\sum_{0 \leq i \leq j} \alpha_i \beta_{j-i}$ cells in dimension $j \leq m$.
\end{cor}

\section{A simple lemma}

\begin{lemma} \label{ovoprincipal} Let $G = F_{n, \infty}$ and $M = \langle x_1, x_2, \ldots , x_n \rangle$. For every non-zero homomorphism  $\chi : G \to {\mathbb R}$ define ${\mathcal A}_{\chi}$ as the set of all subgroups of finite index in $G$ that contain $Ker (\chi)$. Let $H \in {\mathcal A}_{\chi}$ and $\alpha$ be the smallest positive integer such that $x_0^{\alpha} \in H$. Define $ B = \langle x_0^{\alpha}, \cup_{j \in \mathbb{Z}} (H \cap M)^{x_0^j} \rangle$. Then

a) $B$ is an HNN extension with stable letter $x_0^{\alpha}$ and a base group $T = H \cap M $ and associated subgroup
$T $ and $T^{x_0^{\alpha}}$;

b) $H / B$ is finite cyclic, hence $$
d(H) \leq d(B) + 1 \leq d(H \cap M) + 2;
$$

c) there is a $K(H, 1)$ with $r(H, j)$ cells in dimension $j$, there is a $K(B, 1)$ with $r(B, j)$ cells in dimension $j$ and  there is a $K(T, 1)$ with $r(T, j)$ cells in dimension $j$   such that
$$
r(H, j) =  \sum_{0 \leq i \leq j} r(B, i)
$$
and
$$
r(B, j) = r(T, j) + r(T, j-1);
$$

d) fix the isomorphism $\theta : M \to G$ sending $x_i$ to $x_{i-1}$ for all $i \geq 1$ and define $\rho  = (\chi |_{M})  \theta^{-1} : G \to {\mathbb R}$.
Then
$$
\theta (H \cap M) \in {\mathcal A}_{\rho}.
$$
\end{lemma}

\begin{proof}

a) Note that $K =  \cup_{j \in \mathbb{Z}}  M^{x_0^j} $ is a normal subgroup of $G$ with $G / K \simeq \mathbb{Z}$ and $ \cup_{j \geq 0}  M^{x_0^j}  =  M$.
Similarly since $H$ is normal in $G$  we have $\cup_{j \geq  0}  (H \cap M)^{x_0^j} =H \cap M$. Then
\begin{equation} \label{def-B}
B = \langle x_0^{\alpha}, \cup_{j \in \mathbb{Z}} (H \cap M)^{x_0^j} \rangle = \langle x_0^{\alpha}, H \cap M \rangle,
\end{equation}
hence
$B$ is an HNN extension with stable letter $x_0^{\alpha}$ and a base group $T = H \cap M $ and associated subgroup
$T $ and $T^{x_0^{\alpha}}$.

b) Observe that $K \cap H < B \leq H$ and $H / (K \cap H) \simeq \mathbb{Z}$, so $H / B$ is finite cyclic and $d(H) \leq d(B) + 1$. On other hand  by part a) we have $d(B) \leq d(H \cap M) + 1$.

c) By part a) and by Lemma \ref{Ross} for $T = H \cap M$
\begin{equation} \label{ovo1}
r(B, j) = r(T, j) + r( T, j-1).  
\end{equation}
Then 
 since $H/B$ is finite cyclic, there is a $K(H/B, 1)$ with 1 cell in each dimension. Then by Corollary \ref{stackcor} there is a $K(H, 1)$ with $r(H, j)$ cells in dimension $j$ such that
\begin{equation} \label{ovo2}
r(H, j) = \sum_{0 \leq i \leq j} r(B, i) r(H/B, j-i) = \sum_{0 \leq i \leq j} r(B, i) \leq 2 \sum_{0 \leq i \leq j} r(H \cap M,i) ,
\end{equation}

d)  Note that $Ker (\chi) \leq H$ imply  $Ker (\chi |_M) = Ker (\chi) \cap M \leq H \cap M$. Thus $Ker (\rho) = \theta(Ker (\chi |_M)) \leq    \theta(H \cap M)$.

\end{proof}

\section{Proof of Theorem A1}

Set $G = F_{n, \infty}$ and let $\chi : G \to \mathbb{R}$ be a non-zero homomorphism. Define ${\mathcal A}_{\chi}$ as the set of all subgroups of finite index in $G$ that contain $N = Ker(\chi)$.

We aim to show that for every homomorphism $\chi$ we have
\begin{equation} \label{character}
\sup_{H \in {\mathcal A}_{\chi}} d(H) < \infty
\end{equation}

We prove (\ref{character}) by induction on $k = rk(G/ N)$. Thus we can assume that $k \geq 1$ and if $k > 1$ then (\ref{character}) holds for smaller values of $k$.

Observe  that if $N$ is  finitely generated there is nothing to prove. 

Suppose that $N$ is not finitely generated.  By Theorem \ref{preli2} $\Sigma^1(G)^c$ has precisely two points $\{ [\chi_1], [\chi_2] \}$, where  $\chi_1(x_0) =  - 1, \chi_1(x_i) = 0$ for $i \geq 1$ and $\chi_2(x_i) = 1$ for all $i \geq 0$.
Since $N$ is not finitely generated by Theorem \ref{preli1}
$$
\chi_1(N) = 0 \hbox{ or } \chi_2(N) = 0.
$$
By (\ref{ovo40}) $\chi_2$ is obtained from $\chi_1$ by applying an automorphism of $G$, so it is enough to consider the first case i.e. $\chi_1(N) = 0$.

Note that all homomorphisms $\mu : G \to \mathbb{R}$ factor through $G / G' = \mathbb{Z}^n$, hence are given by  $\mu(g) = (\pi(g), v_{\mu})$, where $\pi : G \to G / G' = \mathbb{Z}^n = \mathbb{Z} \overline{x}_0 \oplus \mathbb{Z} \overline{x}_1 \oplus  \ldots  \oplus \mathbb{Z} \overline{x}_{n-1}  $ is the canonical projection, $\overline{x}_i$ is the image of $x_i$ in $G / G'$, $v_{\mu} = (\mu(x_0), \ldots, \mu(x_{n-1})) \in \mathbb{R}^n$ and $( \ \ , \ \  )$ is the standard scalar product in $\mathbb{R}^n$.
Since $\chi_1(N) = 0$  for $v_{\chi} = (v_0, \ldots , v_{n-1}) = (\chi(x_0), \ldots, \chi(x_{n-1}))$ and for every $z_0, \ldots, z_{n-1} \in \mathbb{Z}$ such that $\sum_{0 \leq i \leq n-1} z_i v_i = 0$ we have $z_0 = 0$. Hence $\mathbb{Z} v_0 \cap (\sum_{1 \leq i \leq n-1} \mathbb{Z} v_i) = 0$ and $v_0 \not= 0$, so
\begin{equation} \label{*1}
\mathbb{Z}^k = Im (\chi) = \mathbb{Z} v_0 \oplus (\sum_{1 \leq i \leq n-1} \mathbb{Z} v_i)
\end{equation}
Recall that
$M = \langle x_1,  x_2, \ldots , x_n \rangle < G.
$
Consider
$
\nu = \chi |_{M} : M \to \mathbb{R}.
$
Note that $\chi(x_n) = \chi(x_1^{x_0}) = \chi(x_1)$, hence by (\ref{*1})
\begin{equation} \label{smallerrank} rk \ Im (\nu) = rk ( \sum_{1 \leq i \leq n-1} \mathbb{Z} v_i) = k-1.
\end{equation}
By Lemma \ref{ovoprincipal} d)  $\theta(H \cap M) \subseteq {\mathcal A}_{\rho}$, where 
$ \rho  = \nu \theta^{-1} : G \to {\mathbb R}. $
Then $rk(Im (\rho)) = rk(Im (\nu))$ and by induction 
$$
\sup_{S \in {\mathcal A}_{\rho}} \ d(S) < \infty,$$
so
$$ 
\sup_{H \in {\mathcal A}_{\chi}} \ d(H \cap M)  = \sup_{H \in {\mathcal A}_{\chi}} \ d(\theta(H \cap M)) \leq \sup_{S \in {\mathcal A}_{\rho}} \ d(S)  < \infty.
$$
\noindent
 Consider the group $B = \langle x_0^{\alpha}, \cup_{j \in \mathbb{Z}} (H \cap M)^{x_0^j} \rangle$ defined in Lemma \ref{ovoprincipal} and by Lemma \ref{ovoprincipal} b)
 $d(H) \leq 2 + d(H \cap M)$, hence
$$
\sup_{H \in {\mathcal A}_{\chi}} \ d(H) \leq 2 + \sup_{H \in {\mathcal A}_{\chi}}  \ d(H \cap M ) < \infty.
$$
This completes the proof of (\ref{character}) and (\ref{character}) applied for $Ker (\chi) = G'$ completes the proof of Theorem A1.  

2. Now using more detailed information on the $Aut(F_{n, \infty})$ we find an upper bound on the suprimum in (\ref{character}) for $n \geq 3$. We aim to show that for any subgroup $H$ of finite index in $G$ we have
\begin{equation} \label{ovo5}
d(H) \leq n+2 + d(G' \langle x_0, x_{n-1} \rangle) < \infty.
\end{equation}
Consider $H \in {\mathcal A}_{\chi}$, set $\alpha_i = \chi(x_i)$ and let $B$ be the group defined in Lemma \ref{ovoprincipal}. 

If $\alpha_i = 0$ for $1 \leq i \leq n-1$ then $B = H$ and by Lemma \ref{ovoprincipal} b)  we have  $d(H) \leq 1 + d(H \cap M) = 1 + d(M) = 1+ n$.

If $\alpha_i \not= 0$ for some $1 \leq i \leq n-1$ substituting $\chi$ with $\widetilde{\varphi}^{i-1} (\chi)$, where $\widetilde{\varphi}$ was defined in section \ref{automorphisms}, we can assume that $\alpha_1 \not= 0$, hence we can assume that $\alpha_1 = 1$. 

By Lemma \ref{ovoprincipal} d) we have $\theta(H \cap M) \in {\mathcal A}_{\rho}$.
Thus by Lemma \ref{ovoprincipal} b)
\begin{equation} \label{ovo8}
sup_{H \in {\mathcal A}_{\chi}} d(H) \leq 2 + sup_{H \in {\mathcal A}_{\rho}} d(H). 
\end{equation}

Note that $\rho(x_0) = \rho(x_{n-1}) = \chi(x_1) = 1$.
We can choose from the very beginning  $\chi$ in such a way that $Ker (\chi) = G'$. Then since 
$\rho(x_i) = \chi(x_{i+1})$ for $1 \leq i \leq n-2$ we have
$Ker(\rho) = G' \langle x_0 x_{n-1}^{-1} \rangle$.
Note that if $n \geq 3$ using the matrices $A$ and $C$ from section \ref{automorphisms} there are homomorphisms $\rho_0, \tilde{\rho} : G \to \mathbb{R}$ such that
\[
A^{n-3} C .  \begin{pmatrix} \rho(x_0) = 1 \\ \rho(x_1) \\
                    \cdots \\ \rho(x_{n-2}) \\
                    \rho(x_{n-1}) = 1 \end{pmatrix} =  
A^{n-3} \begin{pmatrix}  \tilde{\rho}(x_0) \\  \tilde{\rho}(x_1) \\
\cdots \\
                    \tilde{\rho}(x_{n-2}) = 0 \\
                    \tilde{\rho}(x_{n-1}) \end{pmatrix} =  
 \begin{pmatrix} \rho_0(x_0) \\
\rho_0(x_1) = 0 \\
\cdots \\
                    \rho_0(x_{n-2}) \\
                    \rho_0(x_{n-1}) \end{pmatrix},\]
 hence
 there is an automorphism of $G$ sending $\rho$ to $\rho_0$ and $Ker (\rho_0) = G' \langle x_1 \rangle$. 
Then 
by (\ref{ovo8}) and using that ${\mathcal A} = {\mathcal A}_{\chi}$ we have
\begin{equation} \label{ovo10}
sup_{H \in {\mathcal A}} d(H) \leq 2 +  sup_{H \in {\mathcal A}_{\rho}} d(H) = 2 +  sup_{H \in {\mathcal A}_{\rho_0}} d(H).
\end{equation}

Assume now that $H \in  {\mathcal A}_{\rho_0}$. Recall that $\rho_0 (x_1) = 0$ and $ rk (Im (\rho_0)) = rk (Im (\rho)) = rk( Im (\chi)) - 1 = n-1$. By (\ref{ovo8}) applied for $\rho_0$ instead of $\chi$ we have
\begin{equation} \label{ovo88}
sup_{H \in {\mathcal A}_{\rho_0}} d(H) \leq 2 + sup_{H \in {\mathcal A}_{\rho_1}} d(H), 
\end{equation}
where $\rho_1 =  (\rho_0 |_{M})  \theta^{-1}$. Thus $\rho_1(x_0) = \rho_1 (x_{n-1}) = \rho_0 (x_1) = 0$ and $\rho_1(x_i) = \rho_0(x_{i+1})$ for $ 1 \leq i \leq n-2$. Thus $rk (Im (\rho_1)) = rk (Im (\rho_0)) - 1 =  n-2$ and  
$Ker(\rho_1) = G' \langle x_0, x_{n-1} \rangle \not\subseteq Ker (\chi_1) \cup Ker (\chi_2)$, so $Ker(\rho_1)$ is finitely generated.

Finally  for $H \in {\mathcal A}_{\rho_1}$ we have that $H / G' \langle x_0, x_{n-1} \rangle \simeq \mathbb{Z}^{n-2}$, so
\begin{equation} \label{ovo11}
d(H) \leq d(G' \langle x_0, x_{n-1} \rangle) + n - 2.
\end{equation}
Then by  (\ref{ovo10}), (\ref{ovo88}) and (\ref{ovo11}) for $n \geq 3$
$$
sup_{H \in {\mathcal A}} d(H)  \leq 2 + 2 + d(G' \langle x_0, x_{n-1} \rangle) + n - 2 = d(G' \langle x_0, x_{n-1} \rangle ) + n + 2.
$$

Note that Theorem A2 gives an upper bound on $d(H)$ when $n = 2$.

\section{Proof of Theorem A2}
Let $H \in {\mathcal A}_{\chi_0}$ for some non-zero homomorphism 
$
\chi_0 : F \to \mathbb{R}
$
 i.e. 	$Ker (\chi_0) \subseteq  H$ and $H$ has finite index in $F$. There are several cases.

1. Suppose that $\chi_0 \in {\mathbb R} \chi_1$ i.e. $\chi_0(x_1) = 0$. Then $M \subseteq Ker (\chi_0) \subseteq H$ and for the group $B$ defined in Lemma \ref{ovoprincipal} we have $H = B$. Then by Lemma \ref{ovoprincipal} c) there is a $K(H, 1)$ complex with
$$
r(H, j) = r(M, j) + r(M, j-1)
$$
cells in dimension $j$, where $r(M,j)$ is the number of $j$-dimensional cells in some $K(M, 1)$ complex. By \cite{BG} we might take $r(M, j) = 2$ for $j \geq 1$ and $r(M, 0) = 1$. In particular $r(H, j) = 4$ for every $j \geq 2$, $r(H, 1) = 3$ and $r(H, 0) = 1$.

2.  Suppose that $\chi_0 \in {\mathbb R} \chi_2$. Then since $\chi_1$ and $\chi_2$ are conjugated by an outer automorphism of $G$ we reduce to case 1 i.e. there is a $K(H, 1)$ complex with $r(H, j) = 4$ cells in dimension $j \geq 2$,  $r(H, 1) = 3$ and $r(H, 0) = 1$.

3. Here we do not impose any restrictions on $\chi_0$. Let $B$ be the group defined in Lemma \ref{ovoprincipal} i.e. $
B = \langle x_0^{\alpha}, \cup_{j \in \mathbb{Z}} (H \cap M)^{x_0^j} \rangle.
$
Then by Lemma \ref{ovoprincipal} c) there is a $K(B,1)$ complex with 
\begin{equation} \label{volume1}
r(B, j) = r(T, j) + r(T, j-1)
\end{equation}
cells in dimension $j$, where $r(T,j)$ is the number of $j$-dimensional cells in some $K(T, 1)$ complex, $T = H \cap M$.
This reduces the problem to the study of the original problem for the finite index subgroup $H \cap M$ of $M$. Note that the restriction $\nu$ of $\chi_0$ to $M$ has the property $\nu(x_1) = \nu(x_2) \not= 0$ and $M \simeq F$ via $x_i$ does to $x_{i-1}$. Thus we can apply case 2 for the finite index subgroup $T = H \cap M$ of $M$ and obtain that  
there is a $K(T, 1)$ complex with $r(T, j) = 4$ cells in dimension  $j \geq 2$,  $r(T, 1) = 3$ and $r(T, 0) = 1$.
Then by (\ref{volume1}) we have 
$r(B, j) = 8$ for $j \geq 3$, $r(B, 2) = 7$, $r(B, 1) = 4$ and $r(B, 0) = 1$.

Finally note that by Lemma \ref{ovoprincipal}  there is a $K(H, 1)$ with $r(H, j)$ cells in dimension $j$ such that
$
r(H, j) = \sum_{0 \leq i \leq j} r(B, i),
$
hence
$$
r(H, j) = \sum_{0 \leq i \leq j} r(B, i) = 1 + 4 + 7 + (j-2) 8 = 8j - 4 \hbox{ for } j \geq 3, 
 $$
$$
r(H,2) = 12, r(H, 1) = 5 \hbox{ and } r(H, 0) = 1.
$$

\section{One Sigma condition}

\begin{prop} \label{sigma}  Let $G = F_{n, \infty}$ and let ${\mathcal A}$ be the set of all subgroups of finite index in $G$. Suppose that 
\begin{equation} \label{sigmacondition}
\Sigma^m(G) = S(G) \setminus  conv_{\leq 2} \{ [\chi_1], [\chi_2] \} 
\end{equation}
Then for every $H \in \mathcal A$ there is a $K(H, 1)$ CW-complex with $r(H,j)$ cells in each dimension $0 \leq j \leq m$ such that $$\sup_{H \in {\mathcal A}} r(H,j) < \infty \hbox{ for }  0 \leq j \leq m.$$
\end{prop}

{\bf Remarks} 

1. By Theorem  \ref{preli3} (\ref{sigmacondition}) holds for $m = 2$. Then if (\ref{sigmacondition}) holds for  some $m > 2$ since $\Sigma^{j}(G) \subseteq \Sigma^{j-1}(G)$ for all $j \geq 2$ we have
$$
\Sigma^j(G) = S(G) \setminus  conv_{\leq 2} \{ [\chi_1], [\chi_2] \} \hbox{ for all } 2 \leq j \leq m. $$

2. The proof of Proposition \ref{sigma} requires only that 
\begin{equation} \label{eq123}
S(G) \setminus \{ [r_1 \chi_1 + r_2 \chi_2] | (r_1, r_2) \in \mathbb{R} \times \mathbb{R} \setminus (0,0) \} \subseteq \Sigma^m(G)
\end{equation}
but by \cite[Thm.~10]{K} the conditions (\ref{eq123}) and (\ref{sigmacondition}) are equivalent.

\begin{proof}

Consider   a non-zero homomorphism $\chi : G \to \mathbb{R}$ and recall that ${\mathcal A}_{\chi}$ is the set of all subgroups of finite index in $G$ that contain $N = Ker(\chi)$.

We claim that
 \begin{equation} \label{ovo123} \sup_{H \in {\mathcal A}_{\chi}} r(H,j) < \infty \hbox{ for }  0 \leq j \leq m.\end{equation}

Proposition \ref{sigma} follows from (\ref{ovo123}) applied for a homomorphism $\chi$ such that $Ker (\chi) = G'$.

\noindent
To prove the claim we induct on $k =rk(G/N)$ and assume that $k \geq 1$ and if $k > 1$ then the claim holds for  smaller values of $k$.

1. Suppose first that $N$ has type $F_m$, so there is a $K(N, 1)$ CW complex with $\alpha_i$ cells in dimension $0 \leq i \leq m$. Observe that for $H \in {\mathcal A}_{\chi}$ we have $H / N \simeq \mathbb{Z}^k$, hence there is a $K(H/ N, 1)$ CW complex of dimension $k$ and with $k \choose j$ cells in dimension $0 \leq j $. Hence by Corollary \ref{stackcor} there is a $K(H, 1)$ CW complex with $r(H,j) = \sum_{0 \leq i \leq j} \alpha_i {k \choose j-i }$ cells in dimension $j \leq m$.

2. Consider the group $B$ from Lemma \ref{ovoprincipal}. Then by Lemma \ref{ovoprincipal} c) there is a $K(H \cap M, 1)$  complex with $r(H \cap M, j)$ cells in dimension $j$ and a $K(B, 1)$  complex with $r(B, j)$ cells in dimension $j$ such that 
\begin{equation} \label{ovo1}
r(B, j) = r(H \cap M, j) + r( H \cap M, j-1)  
\end{equation}
and there is a $K(H, 1)$ with $r(H, j)$ cells in dimension $j$ such that
\begin{equation} \label{ovo2}
r(H, j) =  \sum_{0 \leq i \leq j} r(B, i) \leq 2 \sum_{0 \leq i \leq j} r(H \cap M,i) ,
\end{equation}
Thus it suffices to show that
\begin{equation} \label{ovo4}
\sup_{H \in {\mathcal A}_{\chi}} r(H \cap M,j) < \infty \hbox{ for }j \leq m.
\end{equation}
Consider $\rho$ from Lemma \ref{ovoprincipal} d) and note that 
 $\chi(x_n) = \chi(x_1^{x_0}) = \chi(x_1)$,  hence $\rho(x_0) = \rho(x_{n-1})$. By Lemma \ref{ovoprincipal} d)  $S = \theta(H \cap M) \subseteq {\mathcal A}_{\rho}$
and so
\begin{equation} \label{nnn} \sup_{H \in {\mathcal A}_{\chi}} r(H \cap M, j) = \sup_{H \in {\mathcal A}_{\chi}} r(\theta(H \cap M), j) \leq	\sup_{S \in {\mathcal A}_{\rho}} r(S, j)   < \infty,
\end{equation}
where $r(H \cap M, j)$, $ r(\theta(H \cap M), j) $, $r(S, j)$ are the numbers of $j$-cells in some $K(H \cap M, 1)$, $K(\theta(H \cap M), 1)$ and $K(S, 1)$-complexes.

2.1. Suppose  that $\chi_1(N) = 0$ and recall that $N = Ker (\chi)$. We can continue as in the proof of Theorem A1. By    (\ref{smallerrank}) we have $rk \ Im (\nu) = k-1$, where $\nu = \chi |_M$, so by induction there is a $K(S,1)$ complex with $r(S,j)$ cells in dimension $j \leq m$ such that 
$$
\sup_{S \in {\mathcal A}_{\rho}} r(S, j) < \infty. 
$$
Then (\ref{nnn}) completes the proof.

2.2. Suppose that $M \subseteq H$. In this case $B = H$  and by (\ref{ovo1}) $r(B, j) = r(M, j) + r( M, j-1) < \infty$. 

2.3. Finally we reduce the general case  to the previous cases.
 
By (\ref{nnn}), condition (\ref{ovo4}) is equivalent to the original claim for $\chi(x_0) = \chi(x_{n-1})$ and we are left to prove the proposition for  $\chi(x_0) = \chi(x_{n-1})$ and $N = Ker(\chi)$.
By case 1 we can assume that $N$ is not of type $F_m$. Then by Theorem \ref{preli1} there is $\chi_0 : G \to \mathbb{R}$ such that $\chi_0(N) = 0$ and $[\chi_0] \in S(G) \setminus \Sigma^m(G) =  conv_{\leq 2} \{ [\chi_1], [\chi_2] \}$, thus $\chi_0  = (r_2 - r_1, r_2, \ldots , r_2)$ for some positive real numbers  $r_1, r_2$ i.e. $\chi_0( x_0) = r_2 - r_1$ and $\chi_0(x_i) = r_2$ for $1 \leq i \leq n-1$. Note that $x_0^{-1} x_{n-1} \in Ker (\chi) = N \subseteq Ker (\chi_0)$, hence $0 = \chi_0(x_0^{-1} x_{n-1}) = r_2 - (r_2 - r_1) = r_1$, hence $\chi_0 =  r_2 \chi_2$ and we can assume that $r_2 = 1$. Furthermore by applying the automorphism $\mu$ of $F_{n, \infty}$ that swaps $\chi_1$ with $\chi_2$ we can reduce to the case when $\chi_0 = \chi_1$, then apply case 2.1. 
\end{proof}

\section{ Proof of Theorem B}
Note that by Theorem \ref{preli3} the condition on the $\Sigma^m(G)$ in the statement of Proposition \ref{sigma}  holds for $m = 2$.

By Proposition \ref{sigma} there is a $K(G_s, 1)$ complex with $r(G_s, j)$ cells in dimension $j$ for $0 \leq j \leq 2$ such that $sup_s r(G_s, j) < \infty$, so
\begin{equation} \label{vienna}
\lim_{s \to \infty} \ r(G_s,j) / [G : G_s] = 0 \hbox{ for } 0 \leq j \leq 2.
\end{equation}
From the 2-skeleton of the $K(G_s, 1)$ complex we get a group presentation with $r(G_s, 1) - r(G_s, 0) + 1$ generators and $r(G_s, 2)$ relators. Then 
by \cite[Lemma~2]{BT} (observe the definition of deficiency in \cite{BT} is the infimum of $|R| - |X|$, so in the following we adapt \cite[Lemma~2]{BT} to our definition of deficiency)
\begin{equation} \label{*10}
1- r(G_s, 0) +  r(G_s,1) - r(G_s,2) \leq def(G_s) \leq rk( H_1(G_s, \mathbb{Z})) - d(H_2(G_s, \mathbb{Z})) $$ $$ \leq rk( H_1(G_s, \mathbb{Z})).
\end{equation}
Observe that $G'$ is a simple group and the commutator $G_s'$ is normal in $G$, hence $G' = G_s'$ and so \begin{equation} \label{vienna2}   rk( H_1(G_s, \mathbb{Z})) =  rk( H_1(G, \mathbb{Z})) = 2.\end{equation}
By (\ref{vienna}), (\ref{*10}) and (\ref{vienna2})
$$
\lim_{s \to \infty} def(G_s) / [G : G_s] = 0,
$$
so Theorem B holds.

\section{Higher dimensional gradients}

\subsection{Preliminaries} \label{final0}

Let $G$ be a group of type $F_m$ that has a $K(G,1)$ CW-complex with $\alpha_j$ cells in dimension $j \leq m$.
Define 
$$\chi_m(G) =   inf \{ \sum_{0 \leq i \leq m} (-1)^{m - i} \alpha_i \   \}.$$
This is well defined if for any such CW-complex we have 
\begin{equation} \label{*}
\sum_{0 \leq i \leq m} (-1)^{m - i} \alpha_i \geq 0.
\end{equation}

We recall the definition of the Novikov ring  $\widehat{\mathbb{Z} G}_{\chi}$ associated to a non-zero homomorphism $\chi : G \to \mathbb{R}$. By definition 
an element of  $\widehat{\mathbb{Z} G}_{\chi}$ is an element $\lambda = \sum_{g \in G} z_g g$ of $\prod_{g \in G} \mathbb{Z} g$ (thus the sum can be infinite) with the property that $supp(\lambda) \cap (- \infty, r]$ is finite for every $r \in \mathbb{R}$.  The Novikov ring can be viewed as a completion of the the group algebra $\mathbb{Z} G$ and has strong links with $\Sigma$ theory. The following can be found in an appendix of the unpublished book on $\Sigma$-theory \cite{BieriStrebelbook}. As the proof is easy we give a sketch.

\begin{lemma} \cite{BieriStrebelbook} Let $G$ be a group of type $FP_m$, where the trivial $\mathbb{Z} G$-module $\mathbb{Z}$ has  a free resolution
$$
{\mathcal F} : \ldots \to \mathbb{Z} G ^{\alpha_i}   \stackrel{\partial_i}{\rightarrow}  \mathbb{Z} G ^{\alpha_{i-1}} 
\stackrel{\partial_{i-1}}{\rightarrow} \ldots \stackrel{\partial_1}{\rightarrow}
\mathbb{Z} G ^{\alpha_{0}} \stackrel{\partial_0}{\rightarrow} \mathbb{Z} \to 0
$$
Let $\chi : G \to \mathbb{R}$ be a non-zero character and $\widehat{\mathbb{Z} G}_{\chi}$ be the Novikov ring associated with $\chi$.  Suppose that $Tor^{\mathbb{Z} G}_i( \widehat{\mathbb{Z} G}_{\chi}, \mathbb{Z}) = 0$ for all $i \leq m-1$. Then condition (\ref{*}) holds.
\end{lemma}

\begin{proof} Applying the functor $  \widehat{\mathbb{Z} G}_{\chi} \otimes_{\mathbb{Z} G}$ to the free resolution $\mathcal F$ we obtain an exact sequence  up to dimension $m$ 
$$
{\mathcal S} : \widehat{\mathbb{Z} G}_{\chi} ^{\alpha_m}   \stackrel{\hat{\partial}_m}{\rightarrow}  \widehat{\mathbb{Z} G}_{\chi} ^{\alpha_{m-1}} 
\stackrel{\hat{\partial}_{m-1}}{\rightarrow} \ldots \stackrel{\hat{\partial}_1}{\rightarrow}
\widehat{\mathbb{Z} G}_{\chi} ^{\alpha_{0}} \stackrel{\hat{\partial}_0}{\rightarrow} \widehat{\mathbb{Z} G}_{\chi} \otimes_{\mathbb{Z} G} \mathbb{Z}  = 0
$$
Then the short exact sequences associated to the above exact sequence give that $Im (\hat{\partial}_i)$ is projective 
$\widehat{\mathbb{Z} G}_{\chi}$-module for all $0 \leq i \leq m-1$, all these short exact sequences split. We apply to the split short exact sequences the functor $R \otimes_{\widehat{\mathbb{Z} G}_{\chi}} - $, where $R$ is the field of fraction of $\widehat{\mathbb{Z} Q}_{\chi}$ and $Q$ is the maximal torsion-free abelian quotient of $G$, and obtain short exact sequences again. 
Hence the sequence 
$$R \otimes_{\widehat{\mathbb{Z} G}_{\chi}} {\mathcal S} :
R^{\alpha_m}  \rightarrow R^{\alpha_{m-1}} \to \ldots \to R^{\alpha_0} \to 0
$$ is exact and since $R$ is a field,  (\ref{*}) holds.
\end{proof}

The following result for $m = 1$ is due to Sikorav \cite{S}.

\begin{theorem} \cite{S}, \cite[Thm.~2]{Bieri} Let $G$ be a group of type $FP_m$. Then $[\chi] \in \Sigma^m(G, \mathbb{Z})$ if and only if $Tor^{\mathbb{Z}G}_i( \widehat{\mathbb{Z} G}_{\chi}, \mathbb{Z}) = 0$ for $i \leq m$.
\end{theorem}

We use the above results to deduce the following corollary.

\begin{cor} Let $G$ be a subgroup of finite index in the generalised Thompson group $F_{n, \infty}$. Then $\Sigma^{\infty}(G, \mathbb{Z}) \not= \emptyset$. In particular for $[\chi] \in \Sigma^{\infty}(G, \mathbb{Z})$ we have $Tor^{\mathbb{Z}G}_i( \widehat{\mathbb{Z} G}_{\chi}, \mathbb{Z}) = 0$ for all $i \geq 0$, hence condition (\ref{*}) holds for $G$  for every $m$.
\end{cor}

\begin{proof} Since $F_{n, \infty}$ has type $FP_{\infty}$ the points antipodal to the points of  $\Sigma^1(F_{n, \infty}, \mathbb{Z})^c$ are in $\Sigma^{\infty}(F_{n, \infty}, \mathbb{Z})$ \cite[Prop.~4.2]{Meinert}, \cite[Thm.~2.1]{BGK}. Hence $\Sigma^{\infty}( F_{n, \infty}, \mathbb{Z}) \not= \emptyset$. Note that since $G$ and $F_{n, \infty}$ have the same commutator every  homomorphism $\chi : G \to \mathbb{R}$ extends to a homomorphism $F_{n, \infty} \to \mathbb{R}$. Then by a result first obtained by Bieri and Strebel (1987 unpublished) and published in \cite{Sc} $\Sigma^{\infty}(G, \mathbb{Z}) = \Sigma^{\infty}(F_{n, \infty}, \mathbb{Z}) \not= \emptyset$ and we can apply the previous two results. 
\end{proof}

\subsection{Proof of Theorem C} 
Theorem C follows directly from Theorem A2. Alternative approach is to apply Proposition \ref{sigma},
 note that (\ref{sigmacondition}) holds by Theorem \ref{preli3}.

\end{document}